\documentclass[12pt,leqno,letterpaper]{article}

\usepackage{amsmath,amsthm,enumerate,amssymb}
\usepackage[utf8]{inputenc}

\usepackage[english]{babel}
\usepackage{graphicx}
\usepackage{color}

\usepackage{graphicx}

\newtheorem{theorem}{Theorem}[section]

\newtheorem{example}[theorem]{Example}

\DeclareMathOperator{\frechetdiff}{\mathit d}
\newcommand{\fd}[1]{\hskip-0.2em\frechetdiff\hskip -0.3em{#1}}
\newcommand{\closefd}[1]{\hskip-0.2em\frechetdiff\hskip -0.35em{#1}}

\begin{document}

\title{A note on the parallel sum}
\author{Frank Hansen}
\date{July 16, 2021\\\small{Revised November 16, 2021}}

\maketitle

\begin{abstract} By using a variational principle we find a necessary and sufficient condition for an operator to majorise the parallel sum of two positive definite operators. This result is then used as a vehicle to create new operator inequalities involving the parallel sum.\\[1ex]
{\bf MSC2010} classification: 47A63\\[1ex]
{\bf{Key words and phrases:}}  the parallel sum; operator inequality.
\end{abstract}

\section{Introduction}
Anderson and Duffin defined the parallel sum $ A:B $ of two positive definite operators $ A $ and $ B $ by setting
\[
A:B=\frac{1}{A^{-1}+B^{-1}}\,,
\]
and they proved  \cite[Lemma 18]{anderson:1969} that for any vector $ \xi $ the inner product
\begin{equation}\label{variational principle}
\bigl((A:B)\xi\mid \xi\bigr)=\inf_\eta\bigl\{(A\eta\mid\eta\bigr) +(B(\xi-\eta)\mid\xi-\eta)\bigr\}.
\end{equation}
We begin by giving an intuitive proof of the variational result in (\ref{variational principle}). The purpose of this note is then to establish that the operator inequality 
\[
A:B \le H
\]
is valid, if and only if there exists an operator $ C $ such that
\[
H=C^*AC+(I-C^*)B(I-C). 
\]
This result then functions as a generator of operator inequalities involving the parallel sum. We refer to \cite{Tian:2020} for a recent paper on the parallel sum.

\section{Preliminaries}

We first establish the rule of differentiating an expectation value with respect to a vector,
\[
\fd{}_x(Ax\mid x)\xi=2\text{Re}(Ax\mid\xi).
\]
Indeed,
\[
\begin{array}{l}
\displaystyle\fd{}_x(Ax\mid x)\xi=\lim_{\varepsilon\to 0}\frac{1}{\varepsilon}\Bigl(\bigl(A(x+\varepsilon\xi)\mid x+\varepsilon\xi\bigr)-(Ax\mid x)\Bigr)\\[2ex]
\displaystyle
=\lim_{\varepsilon\to 0}\frac{1}{\varepsilon}\Bigl(\varepsilon (Ax\mid\xi)+\varepsilon (A\xi\mid x)+\varepsilon^2 (A\xi\mid\xi)\Bigr)=2\text{Re}(Ax\mid\xi).
\end{array}
\]
Let $ A,B $ be positive definite matrices and consider to a given vector $ x $ the vector function
\[
f(\xi)=\bigl(A\xi\mid \xi\bigr)+\bigl(B(x-\xi)\mid x-\xi\bigr).
\]
It is manifestly convex with derivative
\[
\begin{array}{rl}
\closefd{} f(\xi)\eta&=2\text{Re}\bigl(A\xi\mid\eta\bigr)-2\text{Re}\bigl(B(x-\xi)\mid\eta\bigr)\\[1.5ex]
&=2\text{Re}\bigl(A\xi-B(x-\xi)\mid\eta\bigr).
\end{array}
\]
The derivative vanishes in all $ \eta $ if and only if
\[
A\xi-B(x-\xi)=0\qquad\text{or}\qquad (A+B)\xi=Bx,
\]
and this is equivalent to
\begin{equation}\label{formula for xi}
\xi=(A+B)^{-1}Bx.
\end{equation}
In addition,
\[
\begin{array}{rl}
x-\xi&=x-(A+B)^{-1}Bx=(A+B)^{-1}\bigl((A+B)x-Bx\bigr)\\[1.5ex]
&=(A+B)^{-1}Ax.
\end{array}
\]
We thus obtain that
\[
\bigl(A\xi\mid \xi\bigr)=\bigl(A(A+B)^{-1}Bx\mid(A+B)^{-1}Bx\bigr)
\]
and
\[
\bigl(B(x-\xi)\mid x-\xi\bigr)=\bigl(B(A+B)^{-1}Ax\mid(A+B)^{-1}Ax\bigr).
\]
Since $ f $ is convex the global minimum of $ f $ is obtained in $ \xi $ with minimum value
\[
f(\xi)=\bigl(A\xi\mid \xi\bigr)+\bigl(B(x-\xi)\mid x-\xi\bigr).
\]
Since 
\[
B(A+B)^{-1}A=(A^{-1}+B^{-1})^{-1}=A(A+B)^{-1}B,
\]
we calculate the global minimum value to be
\[
\begin{array}{rl}
f(\xi)&=\bigl((A^{-1}+B^{-1})^{-1}x\mid (A+B)^{-1}Bx+(A+B)^{-1}Ax\bigr)\\[2ex]
&=\bigl((A^{-1}+B^{-1})^{-1}x\mid x\bigr)=\bigl((A:B)x \mid x\bigr),
\end{array}
\]
where $ A:B $
is the parallel sum of $ A $ and $ B. $ It is also half of the harmonic mean.
In conclusion, we recover (\ref{variational principle})
and obtain the inequality
\[
\bigl((A:B)x\mid x\bigr)=f(\xi)\le f(\eta)
\]
for any other vector $ \eta. $ For an arbitrary operator $ D $ we set $ \eta=D\xi $ and obtain
\[
\begin{array}{l}
\bigl((A:B)x\mid x\bigr)\le f(D\xi)\\[1.5ex]
=\bigl(AD\xi\mid D\xi\bigr)+\bigl(B(x-D\xi)\mid x-D\xi\bigr)\\[1.5ex]
=\bigl(AD(A+B)^{-1}Bx\mid D(A+B)^{-1}Bx\bigr)\\[1ex]
\hskip 6em+\bigl(B(x-D(A+B)^{-1}Bx)\mid x-D(A+B)^{-1}Bx\bigr),
\end{array}
\]
where we used (\ref{formula for xi}). Putting $ C=D(A+B)^{-1}B $ this is equivalent to
\[
\bigl((A:B)x\mid x\bigr)\le \bigl(C^*ACx\mid x\bigr)+\bigl((I-C^*)B(I-C)x\mid x\bigr).
\]
We have thus proved the following result.

\begin{theorem}\label{generator of operator inequalities}
Let $ A $ and $ B $ be positive definite operators. Then
\[
A:B\le C^*AC+(I-C^*)B(I-C) 
\]
for an arbitrary operator  $ C. $
\end{theorem}

We next investigate the range of the operator function
\[
F(C)=C^*AC+(I-C^*)B(I-C)
\]
to given positive definite operators $ A $ and $ B. $ We consider the operator equation $ F(C)=H $
and rewrite the equation as
\[
C^*(A+B)C+B-C^*B-BC=H.
\]
By multiplying with $ (A+B)^{-1/2} $ from the left and from the right the equation is equivalent to
\[
\begin{array}{l}
(A+B)^{-1/2}C^*(A+B)C(A+B)^{-1/2}+(A+B)^{-1/2} B (A+B)^{-1/2} \\[1.5ex]
-(A+B)^{-1/2} C^*B(A+B)^{-1/2} -(A+B)^{-1/2} BC(A+B)^{-1/2}\\[1.5ex]
=(A+B)^{-1/2} H(A+B)^{-1/2}. 
\end{array}
\]
We now set
\[
X=(A+B)^{1/2}C(A+B)^{-1/2}\quad\text{and}\quad Y=(A+B)^{-1/2}B(A+B)^{-1/2}
\]
and rewrite the equation as
\[
X^*X+Y-X^*Y-YX
=(A+B)^{-1/2} H (A+B)^{-1/2},
\]
which again may be written as
\[
(X-Y)^*(X-Y)-Y^2+Y
=(A+B)^{-1/2} H (A+B)^{-1/2}
\]
or
\[
\begin{array}{l}
(X-Y)^*(X-Y)
=(A+B)^{-1/2} H (A+B)^{-1/2}+Y^2-Y\\[1.5ex]
=(A+B)^{-1/2}\bigl(H-B+B(A+B)^{-1}B\bigr)(A+B)^{-1/2}\\[1.5ex]
=(A+B)^{-1/2}\bigl(H-B(A+B)^{-1}(A+B-B)\bigr)(A+B)^{-1/2}\\[1.5ex]
=(A+B)^{-1/2}\bigl(H-(A:B)\bigr)(A+B)^{-1/2}.
\end{array}
\]
The equation can thus be solved if and only if
\[
H\ge A:B.
\]
Under this condition we may find positive definite solutions in $ X $ given by
\[
X=Y+\Bigl((A+B)^{-1/2}\bigl(H-(A:B)\bigr)(A+B)^{-1/2}\Bigr)^{1/2}
\]
and then obtain
\[
\begin{array}{l}
C=(A+B)^{-1/2}X(A+B)^{1/2}=
(A+B)^{-1/2}Y(A+B)^{1/2}\\[1ex]
+\,(A+B)^{-1/2} \Bigl((A+B)^{-1/2}\bigl(H-(A:B)\bigr)(A+B)^{-1/2}\Bigr)^{1/2}(A+B)^{1/2}\\[1.5ex]
=(A+B)^{-1}B +\Bigl((A+B)^{-1}\bigl(H-(A:B)\bigr)\Bigr)^{1/2}.
\end{array}
\]
Note that the operator appearing inside the square root in the last formula line may not be self-adjoint. It is however similar to a positive semi-definite operator and therefore has a unique square root with positive spectrum.
We have obtained.
\begin{theorem}
Let $ A,B $ and $ H $ be positive definite operators. The operator equation
\[
 F(C)=C^*AC+(I-C^*)B(I-C) =H
\]
has solutions in $ C $ if and only if $ H\ge A:B. $ One of the solutions is then given by
\[
C=(A+B)^{-1}B +\Bigl((A+B)^{-1}\bigl(H-(A:B)\bigr)\Bigr)^{1/2}.
\]
\end{theorem}

\section{Generating operator inequalities}

Theorem~\ref{generator of operator inequalities} may serve as a generator for operator inequalities by suitably choosing the operator $ C. $ For $ C=\lambda I, $ where $ 0\le\lambda\le 1, $ we obtain
\[
A: B\le\lambda^2 A + (1-\lambda)^2 B.
\]
By setting $ \lambda=0, $ $ \lambda=1/2 $ or $ \lambda=1 $ we obtain the well-known inequalities
\[
 A:B\le B,\qquad A:B\le\frac{A+B}{4}\,,\qquad A:B\le  A.
\]
Setting $ C=(A+B)^{-1}B $ we obtain equality
\[
A:B=F(C). 
\]
Indeed, we note that
\[
I-C=I-(A+B)^{-1}B=(A+B)^{-1}(A+B-B)=(A+B)^{-1}A.
\]
Therefore,
\[
\begin{array}{l}
F(C)
=B(A+B)^{-1}A(A+B)^{-1}B+A(A+B)^{-1}B(A+B)^{-1}A\\[1.5ex]
=B(A+B)^{-1}A(A+B)^{-1}B+A(A+B)^{-1}A(A+B)^{-1}B\\[1.5ex]
=\bigl(B^{-1}+A^{-1}\bigr)^{-1}=A:B.
\end{array}
\]
We next use Theorem~\ref{generator of operator inequalities} to obtain new operator inequalities.

\begin{theorem}\label{new operator inequalities}
Let $ A,B $ be positive definite operators.

\begin{enumerate}[(i)]

\item Let $ P $ be an orthogonal projection. We obtain the inequality
\[
A:B\le PAP+(I-P)B(I-P).
\]
Setting $ A=B $ it reduces to the familiar inequality
\[
\frac{1}{2}A\le PAP+(I-P)A(I-P).
\]

\item The inequality
\[
A:B\le (A+B)^{-1}\bigl(BAB+ABA\bigr)(A+B)^{-1}
\]
is valid, and it is strict, since for $ A=B $ it reduces to $ \frac{1}{2}A\le\frac{1}{2}A. $

\item Let $ p $ be a real number. We obtain the inequality
\[
A:B\le (A^p:B^p)(A^{2p-1}:B^{2p-1})^{-1}(A^p:B^p),
\]
and it reduces to equality for $ p=1. $ The inequality is strict for arbitrary $ p, $ since for $ A=B $ it reduces to $ \frac{1}{2}A\le\frac{1}{2}A. $

\end{enumerate}

\end{theorem}

\begin{proof}
By setting $ C=P $ and applying  Theorem~\ref{generator of operator inequalities} we obtain $ (i). $  By setting $ C=B(A+B)^{-1} $ we obtain $ I-C=A(A+B)^{-1} $ and thus
\[
\begin{array}{l}
C^*AC+(I-C)^*B(I-C)\\[1ex]
=(A+B)^{-1}BAB(A+B)^{-1}+(A+B)^{-1}ABA(A+B)^{-1}
\end{array}
\]
from which $ (ii) $ follows. Finally, we set $ C=(A^p+B^p)^{-1}B^p $ and since $ I-C=(A^p+B^p)^{-1}A^p $  and $ A^p(A^p+B^p)^{-1}B^p=B^p(A^p+B^p)^{-1}A^p $ we obtain
\[
\begin{array}{l}
C^*AC+(I-C)^*B(I-C)\\[1.5ex]
=B^p(A^p+B^p)^{-1}A(A^p+B^p)^{-1}B^p+A^p(A^p+B^p)^{-1}B(A^p+B^p)^{-1}A^p\\[1.5ex]
=A^p(A^p+B^p)^{-1}B^p A^{1-2p}B^p(A^p+B^p)^{-1}A^p\\[1ex]
\hskip 9em +A^p(A^p+B^p)^{-1} B (A^p+B^p)^{-1}A^p\\[1.5ex]
=A^p(A^p+B^p)^{-1}B^p\bigl(A^{1-2p}+B^{1-2p}\bigr)B^p(A^p+B^p)^{-1}A^p\\[1.5ex]
=(A^p:B^p)\bigl(A^{2p-1}:B^{2p-1}\bigr)^{-1}(A^p:B^p)
\end{array}
\]
as desired. This proves $ (iii). $
\end{proof}

By multiplying $ (iii) $ in Theorem~\ref{new operator inequalities} by $ 2 $ we obtain the inequality between harmonic means
\begin{equation}
H_2(A,B)\le H_2(A^p,B^p) H_2\bigl(A^{2p-1},B^{2p-1}\bigr)^{-1}H_2(A^p,B^p)
\end{equation}
for positive definite operators $ A $ and $ B $ and arbitrary $ p\in\mathbf R. $ If we in particular put $ p=1/2 $ we obtain
\begin{equation}
H_2(A,B)\le  H_2(A^{1/2}, B^{1/2})^2.
\end{equation}
This is an improvement of the inequality
\[
H_2(A,B)^{1/2}\le H_2(A^{1/2}, B^{1/2})
\]
which is plain. Indeed, for $ 0\le p \le 1, $ we obtain by operator concavity of the function $ t\to t^p $ the inequality
\begin{equation}
\begin{array}{rl}
H_2(A,B)^p&=\displaystyle\left(\frac{2}{A^{-1}+B^{-1}}\right)^p=\left(\frac{A^{-1}+B^{-1}}{2}\right)^{-p}\\[3ex]
&\le \displaystyle\frac{2}{A^{-p}+B^{-p}}=H_2(A^p,B^p).
\end{array}
\end{equation}
The reverse inequality is obtained for $ -1\le p\le 0 $ and $ 1\le p\le 2 $ by operator convexity. It is interesting to note that the inequality
\begin{equation}\label{conjecture}
H_2(A,B)\le  H_2(A^p, B^p)^{1/p}
\end{equation}
is false for $ p=1/4 $ with counter examples in two-by-two matrices. We conjecture that (\ref{conjecture}) is false for $ 0<p<1/2 $ and true for $ 1/2\le p\le 1. $

\subsection{The power means}

Bhagwat and Subramanian \cite[Section 4]{bhagwat:1978} introduced for $ p>0 $ the power mean
\begin{equation}\label{definition of the power mean}
M_p(A,B)=\left(\frac{A^p+B^p}{2}\right)^{1/p}
\end{equation}
of positive definite operators $ A $ and $ B. $ If $ p\ge 1 $ then the function  $ t\to t^{1/p} $ is operator concave and thus 
\[
M_p(A,B)\ge \frac{A+B}{2}\ge 2(A:B)> A:B.
\]
The parallel sum is thus majorized by the power mean. However, this result can in general not be extended to $ 0<p<1. $

\begin{example}
Consider the two-by-two matrices
\[
A = \begin{pmatrix}
      0.14623 & -0.07525\\
      -0.07525 & 0.03873
      \end{pmatrix},\qquad
B = \begin{pmatrix}
      0.733 & -0.43\\
      -0.43 & 0.2525
      \end{pmatrix}.
\]
$ A $ has approximately eigenvalues $ \{0.184955,  5.00338\cdot 10^{-6}\} $ and $ B $ has
approximately eigenvalues $ \{0.985315, 0.00018522\}, $ so they are positive definite. Setting $ p=1/2 $
the smallest eigenvalue of
\[
\left(\frac{A^{1/2}+B^{1/2}}{2}\right)^2-(A:B)
\]
is approximately $ -1.57101  \cdot 10^{-6}. $

\end{example}


\vfill

{\small
\noindent Frank Hansen: Department of Mathematical Sciences, Copenhagen University, Denmark.\\
Email: frank.hansen@math.ku.dk.
      }

\end{document}